\documentclass[reqno]{amsart}%
\usepackage{amsmath}
\usepackage{amsfonts}
\usepackage{CJK}
\usepackage{eufrak}
\usepackage{amssymb}
\usepackage{graphicx}%
\usepackage{hyperref}
\usepackage{xcolor}

\newcommand{\ra}{\rightarrow}
\newcommand{\mc}{\mathbb{C}}
\newcommand{\mr}{\mathbb{R}}
\newcommand{\m}{\mathcal}

\newcommand{\mg}{\mathfrak{g}}
\newcommand{\bi}{\begin{itemize}}
\newcommand{\ei}{\end{itemize}}

\newtheorem{thm}{Theorem}[section]
\newtheorem{cor}[thm]{Corollary}
\newtheorem{lem}[thm]{Lemma}

\theoremstyle{definition}
\newtheorem{defn}{Definition}[section]
\theoremstyle{Conjecture}

\newtheorem{prob}{Problem}[section]
\theoremstyle{remark}

\theoremstyle{Example}
\newtheorem{exm}{Example}[section]

\newcommand{\be}{\begin{equation}}
\newcommand{\ee}{\end{equation}}
\newcommand{\bs}{\begin{split}}
\newcommand{\es}{\end{split}}
\newcommand{\bea}{\begin{eqnarray}}
\newcommand{\eea}{\end{eqnarray}}
\newcommand{\ben}{\begin{eqnarray*}}
\newcommand{\een}{\end{eqnarray*}}
\newcommand{\bet}{\begin{equation}

\begin{split}}
\newcommand{\eet}{\end{split}
\end{equation}}

\author{Fusheng Deng}
\address{Fusheng Deng: \ School of Mathematical Sciences, University of Chinese Academy of Sciences\\ Beijing 100049, P. R. China}
\email{fshdeng@ucas.ac.cn}
\author{Erlend Forn\ae ss Wold}
\address{Erlend Forn\ae ss Wold: Matematisk Institutt, Universitetet i Oslo, Postboks 1053 Blindern, 0316 Oslo, Norway}
\email{erlendfw@math.uio.no}

\begin{document}
\title[]{Hamiltonian Carleman approximation and the density property for coadjoint orbits}

\begin{abstract}
For a complex Lie group $G$ with a real form $G_0\subset G$, we prove that any
Hamiltionian automorphism $\phi$ of a coadjoint orbit $\mathcal O_0$ of $G_0$ whose connected components are simply connected,
may be approximated by holomorphic $\mathcal O_0$-invariant symplectic automorphism of the corresponding coadjoint orbit of $G$
in the sense of Carleman, provided that $\mathcal O$ is closed.
In the course of the proof, we establish the Hamiltonian density property for closed coadjoint orbits of all complex Lie groups.
\end{abstract}
\maketitle

\section{Introduction}
Let $\omega_0=dx_1\wedge dx_{n+1}+\cdots+dx_n\wedge dx_{2n}$ be the standard symplectic form on $\mr^{2n}$,
and let $\omega=dz_1\wedge dz_{n+1}+\cdots +dz_n\wedge dz_{2n}$ be the standard holomorphic symplectic form on $\mc^{2n}$.
We denote by $\mathrm{Symp}(\mr^{2n},\omega_0)$ the group of smooth symplectic automorphisms of $(\mr^{2n}, \omega)$,
and by $\mathrm{Symp}(\mc^{2n},\omega)$ the group of holomorphic symplectic automorphisms of $(\mc^{2n}, \omega)$ (throughout
this paper, smooth will mean $C^\infty$-smooth).
The following problem was proposed by N. Sibony (private communication):

\begin{prob}\label{mainprob}
Can any element in $\mathrm{Symp}(\mr^{2n},\omega_0)$
be approximated in the sense of Carleman by elements in $\mathrm{Symp}(\mc^{2n},\omega)$ leaving
$\mathbb R^{2n}$ invariant?
\end{prob}

Motivated by this problem, and also connections to physics (see below),
in this article we will consider the analoguous problem in the more general setting of
complexifications of coadjoint orbits of real Lie groups.  For a real or complex
manifold $X$ with a symplectic form $\omega$ we will denote by $\mathrm{Ham}(X,\omega)$
the smooth path-connected component of the identity in $\mathrm{Symp}(X,\omega)$.
We will prove the following.

\begin{thm}\label{thm:carleman}
Let $G$ be a complex Lie group, and let $G_0\subset G$ be a real form.
Let $\mathcal O_0\subset\mathfrak{g}_0^*$ be a coadjoint orbit whose connected components are simply connected,
let $\mathcal O\subset\mathfrak{g}^*$ be the coadjoint orbit containing $\mathcal O_0$,
assume that
$\mathcal O$ is closed,
and let $\omega_0$ (resp. $\omega$) denote
the canonical symplectic form on $\mathcal O_0$ (resp. $\mathcal O$).
Then given $\phi\in\mathrm{Ham}(\mathcal O_0,\omega_0)$, a
positive continuous function $\epsilon(x)$ on $\mathcal O_0$, and $r\in\mathbb N$,
there exists $\psi\in\mathrm{Ham}(\mathcal O,\omega)$ such that
such that $\psi(\mathcal O_0)=\mathcal O_0$ and
$$
||\psi(x)-\phi(x)||_{C^r}<\epsilon(x)
$$
for all $x\in\mr^{2n}$.
\end{thm}
Here the $C^r$-approximation may be obtained with respect to any given
Riemannian metric on the appropriate jet-space of $\mathcal O_0$.  The connection with Problem \ref{mainprob}
is that if $G$ (resp. $G_0$) is the complex (resp. real) Heisenberg group, then there
are coadjoint orbits $(\mathcal O_0,\mathcal O)\approx (\mathbb R^{2},\mathbb C^{2})$
with $(\mathbb R^{2},\mathbb C^{2})$ equipped with the standard symplectic structures above.
The case of an arbitrary $n$ is obtained by taking products. Note that
$\mathrm{Symp}(\mathbb R^{2n},\omega_0)$ coincides with $\mathrm{Ham}(\mathbb R^{2n},\omega_0)$
and $\mathrm{Symp}(\mathbb C^{2n},\omega)$ coincides with $\mathrm{Ham}(\mathbb C^{2n},\omega)$.

For information about Carleman approximation by functions,
please see the recent survey \cite{FFW18}.
Related to Theorem \ref{thm:carleman},
it was proved in \cite{K-W18} that any diffeomorphism of $\mr^k$
can be approximated by holomorphic automorphisms of $\mc^n$ in the Carleman sense,
provided that $k<n$, however, in that case $\mr^k$ was not left invariant.  \

In the course of the proof we will also prove the following, which follows from quite standard
arguments in Anders\'{e}n-Lempert theory, as soon as one considers the right setup.

\begin{thm}\label{thm:Ham desity coad orbits}
All closed coadjoint orbits of a complex Lie group have the Hamiltonian density property.
\end{thm}

The corresponding theorem for $\mc^{2n}$ was proved by F. Forstneri\v{c} \cite{Forstneric96}.
The volume density property of  closed coadjoint orbits of a complex Lie group $G$
in the case that $G$ is an algebraic group is a corollary of Theorem 1.3 in \cite{K-K17} due to
Kaliman and Kutzschebauch.

\medskip

Symplectic manifolds are very important objects in physics.
In a recent program called "Quantization via Complexification" proposed by Gukov and Witten,
the quantization of a symplectic manifold $(M,\omega_0)$ is studied through the
quantization of its complexification $(X,\omega,\tau)$ (\cite{Gukov-Witten09}\cite{Gukov11}\cite{Witten11}).
In general, the physical symmetry on $(M,\omega_0)$ is given by $\mathrm{Symp}(M,\omega_0)$ (resp. $\mathrm{Ham}(M,\omega_0)$),
while that of the complexification $(X,\omega,\tau)$ is given by $\mathrm{Symp}(X, \omega, \tau)$ (resp. $\mathrm{Ham}(X, \omega, \tau)$).
So in physical terminology, our result is transferred to the
following:  no symmetry is broken after complexification.

\medskip

The article is organized as follows. In Section 2 we will give the relevant background on coadjoint oribits.
In Section 3 we start by setting up a general framework for our approximation.
Then we give some results on $C^r$ regularity of flows of families of smooth vector fields and
$C^r$-convergence of consistent algorithms, following the unpublished Diplomarbeit \cite{Schar} by B. Sch\"{a}r at
the University of Bern, and we prove some results of Anders\'{e}n-Lempert type, leaving invariant a totally real center.
Finally we prove the main theorem. In Section 4 we include some standard examples of coadjoint orbits.

\medskip

\textbf{Acknowledgement:}
The first author is partially supported by the NSFC grant 11871451. The second author is supported by the
NRC grant number 240569.
\section{Preliminaries on coadjoint orbits of Lie groups and complexfications}\label{sec:basic reults aboue coad. orb.}

\subsection{Canonical symplectic structures on coadjoint orbits}\label{subsec:canonical symp}
In this subsection we will collect some standard material on co-adjoint orbits (see \cite{Kir05} for more details).
Let $G$ be a real Lie group with Lie algebra $\mathfrak{g}$.
As a vector space, we can identify $\mathfrak g$ with $T_eG$,
the tangent space of $G$ at the identity.
For $g\in G$, the conjugate map from $G$ to itself
given by $x\mapsto gxg^{-1}$ fixes $e$,
and hence induce a linear isomorphism of $\mathfrak g$.
In this manner,
we get a linear representation $\mathrm{Ad}$ of $G$ on $\mathfrak g$,
which is called the adjoint representation of $G$.
The adjoint representation of $G$ induces a representation
of $G$ on the dual space $\mathfrak g^*$ of $\mathfrak g$,
which is called the coadjoint representation of $\mathfrak g$
and will be denoted by $\mathrm{Ad}^*$.
More precisely, the coadjoint representation is defined as
$$\left(\mathrm{Ad}^*(g)\xi, v\right):=\left(\xi, \mathrm{Ad}(g^{-1})v\right), \ g\in G,\ \xi\in\mg^*,\ v\in\mg,$$
where $(\cdot, \cdot)$ denotes the pairing between $\mg^*$ and $\mg$.
The orbits of elements in $\mathfrak{g}^*$ under the action of $G$
are called coadjoint orbits.


A fundamental fact about coadjoint orbits $\mathcal O$ is that
they carry canonical $G$-invariant symplectic structures,
which are defined as follows.
For a vector $v\in \mathfrak g$,
we have that
$\mathrm{Ad}^*$ induces an action of the one-parameter subgroup
$\{\exp(tv); t\in\mr\}$ of $G$ generated by $v$ on $\m O$,
where $\exp:\mathfrak g\ra G$ is the exponential map.
So we get a smooth vector field $X_v$ on $\m O$ whose value at $\xi\in\m O$ is given by
$$
X_v(\xi)=\frac{d}{dt}|_{t=0} \mathrm{Ad}^*(\exp(tv))\xi.
$$
For $\xi\in\m O$, we have that  $\{X_v(\xi); v\in\mathfrak g\}=T_\xi \m O$
since the action of $G$ on $\m O$ is transitive.


The kernel of the linear map $v\mapsto X_v(\xi)$ from $\mg$ to $T_\xi\m O$ is $\mg_\xi$,
which is the Lie algebra of the isotropy subgroup $G_\xi:=\{g\in G; g\xi=\xi\}$ of $G$ at $\xi$.
The symplectic form $\omega$ on $\m O$ is defined by
$$\omega(X_u(\xi), X_v(\xi)):=\xi([u,v]), \ u, v\in\mathfrak g.$$
We check that $\omega$ is a well-defined symplectic form on $\m O$:
\begin{itemize}
\item $\omega$ is well defined.
It suffices to check that for $u\in\mg$ with $X_u(\xi)=0$ we have $\xi([u,v])=0$ for all $v\in\mg$.
If $X_u(\xi)=0$, then $u\in\mg_\xi$, the Lie algebra of the isotropy subgroup $G_\xi$.
Hence $\mathrm{Ad}^*(\exp(tu))\xi=\xi$ for all $t\in\mr$,
which is equivalent to that $(\mathrm{Ad}^*(\exp(tu))\xi, v)=(\xi,v)$ for all $t\in\mr$ and $v\in\mg$.
Note that $(\mathrm{Ad}^*(\exp(tu))\xi, v)=(\xi, \mathrm{Ad}(exp(-tu))v)$,
taking derivative with $t$ at $t=0$ we get $\xi([u,v])=-(\xi, [-u,v])=0$.


\item $\omega$ is nondegenerate.
For $u\in\mg$, we need to show that $\omega(X_u(\xi), X_v(\xi))=0$ for all $v\in\mg$ implies $X_u(\xi)=0$.
By definition,
\begin{equation*}
\begin{split}
\omega(X_u(\xi), X_v(\xi))
&=-\xi([-u,v])\\
&=-\left(\xi, \frac{d}{dt}|_{t=0}\mathrm{Ad}(\exp(-tu))v\right)\\
&=-\left(\frac{d}{dt}|_{t=0}\mathrm{Ad}^*(\exp(tu))\xi,v\right).
\end{split}
\end{equation*}
So $\omega(X_u(\xi), X_v(\xi))=0$ for all $v\in\mg$ implies $\frac{d}{dt}|_{t=0}\mathrm{Ad}^*(\exp(tu))\xi=X_u(\xi)=0$.
\item $\omega$ is closed.

Since all $X_u (u\in\mg)$ generate $T_\xi \m O$ for all $\xi\in\m O$,
and $\omega(X_u, X_v)$ are smooth functions on $\m O$, we have that
$\omega$ is a smooth 2-form on $\m O$.
Letting $u,v,w\in \mg$, we have
\begin{equation*}
\begin{split}
&X_u\omega(X_v(\xi),X_w(\xi))\\
=&\frac{d}{dt}|_{t=0}\omega(\mathrm{Ad}^*(\exp(tu))\xi)\left(X_v(\mathrm{Ad}^*(\exp(tu))\xi), X_w(\mathrm{Ad}^*(\exp(tu))\xi)\right)\\
=&\frac{d}{dt}|_{t=0}\left(\mathrm{Ad}^*(\exp(tu))\xi, [v,w]\right)|\\
=&\frac{d}{dt}|_{t=0}\left(\xi, Ad(\exp(-tu))[v,w]\right)|\\
=&-(\xi,[u,[v,w]]).
\end{split}
\end{equation*}
By a basic formula about exterior differentiation, we have
\begin{equation*}
\begin{split}
& d\omega(X_u,X_v,X_w)\\
=&X_u\omega(X_v,X_w)-X_v\omega(X_u,X_w)+X_w\omega(X_u,X_v)\\
&-\omega([X_u,X_v],X_w)+\omega([X_u,X_w],X_v)-\omega([X_v,X_w],X_u),
\end{split}
\end{equation*}
which vanishes for all $u,v,w\in\mg$ by the Jacobi identity and the above formula.
Hence $d\omega=0$.
\end{itemize}

For $u\in\mathfrak g$, we can view $u$ as a function on $\mathfrak g^*$ and hence a smooth function on the orbit $\m O$.
A basic fact is the following.
\begin{lem}\label{lem:u is potential}
For any $u\in\mg$, the vector field $X_u$ on $\m O$ is Hamiltonian
with respect to $\omega$ and it's potential is $u$ itself.
In particular, the symplectic structure $\omega$ on $\m O$ is invariant under the action of identity component of $G$.
\end{lem}
\begin{proof}
It suffices to show that
$$du(X_v)(\xi)=i_{X_u}\omega(X_v)(\xi)=\xi([u,v])$$
for all $v\in\mg$. Note that
\begin{equation*}
\begin{split}
 &du(X_v)\\
=&\frac{d}{dt}(\mathrm{Ad}^*\left(\exp(tv))\xi, u\right)|_{t=0}\\
=&\left(\xi, \frac{d}{dt}\mathrm{Ad}(\exp(-tv))u|_{t=0}\right)\\
=&\xi([u,v]).
\end{split}
\end{equation*}
\end{proof}

The following lemma, which is a direct corollary of Theorem 2.13 in \cite{Mon-Zip55}, is also useful.

\begin{lem}\label{lem:closed=embed}
For $\xi\in\mg^*$, if the coadjoint orbit $\mathcal O=G\cdot\xi$ is closed in $\mg^*$,
then the canonical map $G/G_\xi\rightarrow \mg^*$ with image $\mathcal O_\xi$ is a proper embedding
and hence $\mathcal O_\xi$ is a closed submanifold of $\mg^*$.
\end{lem}

The above construction starts from a real Lie group.
In the same way, we can start from a complex Lie group
and carry out the same procedure.
Then all coadjoint orbits in $\mg^*$ are complex manifolds
with canonical $G$-invariant holomorphic symplectic forms,
and the statements parallel to Lemma \ref{lem:u is potential} and Lemma \ref{lem:closed=embed} hold.

\subsection{Complexification of coadjoint orbits of real Lie groups.}

Let $G$ be a connected complex Lie group with a real form $G_0$,
\emph{i.e.}, $G_0$ is a Lie subgroup of $G$ (not necessarily closed) and $\mathfrak g=\mathfrak g_0\oplus i\mathfrak g_0$,
where $\mathfrak g$ and $\mathfrak g_0$ are the Lie algebras
of $G$ and $G_0$ respectively.
We want to show that coadjoint orbits of $G$ are complexifications of the corresponding
coadjoint orbits of $G_0$.
We start by introducing the following

\begin{defn}\label{def:complexification}
A complexification of a symplectic manifold $(M^{2n},\omega_0)$ is a triple $(X^{2n},\omega, \tau)$:
\bi
\item $(X,\omega)$ is a holomorphic symplectic manifold,
\item $\tau$ is an anti-holomorphic involution of $X$,
\item $M\hookrightarrow X$ (proper embedding) and $\omega|_M=\omega_0$,
\item $\tau|_M=Id$, $\tau^*\omega=\bar{\omega}$.
\ei
\end{defn}

For a point $\xi\in \mathfrak g^*_0\ (\text{resp.\ $\mathfrak g^*$})$,
the coadjoint orbit through $\xi$ in $\mathfrak g^*_0$ (resp. $\mg^*$)
will be denoted by $\m O^{\mr}_\xi$ (resp. $\m O_\xi$).
The canonical symplectic form on $\m O^{\mr}_\xi$ (defined in \S \ref{subsec:canonical symp}) is denoted by $\omega_0$,
and the canonical holomorphic symplectic form on $\m O_\xi$ is denoted by $\omega$.
For a point $\xi\in \mathfrak g^*_0$,
we can also view $\xi$ as an element in $\mg^*$.
Then it is clear that $\m O^{\mr}_\xi\subset \m O_\xi$ and and $\omega|_{\m O^\mr_\xi}=\omega_0$.

We will focus on closed orbits.
In the case that $G_0$ is reductive,
the adjoint representation of $G_0$ (resp. $G$) on $\mg_0$ (resp. $\mg$)
and the coadjoint representation $G_0$ (resp. $G$) on $\mg^*_0$ (resp. $\mg^*$)
are isomorphic, and hence the adjoint orbits and coadjoint orbits coincide.
For $\xi\in \mg_0$, the closedness of $\m O^{\mr}_\xi$ in $\mg$ is equivalent
to the closedness of $\m O_\xi$ in $\mg$,
and they are equivalent to that $\xi$ is semisimple,
that is $Ad_\xi$, viewed as an operator on both $\mg _0$ and $\mg$,
is diagonalizable.
In the general case, we will show that if $\m O_\xi$ is closed then $\m O^{\mr}_\xi$ is also closed.
We will prove the following.

\begin{thm}\label{thm:complexification of real orbit}
If $\xi\in \mg^*_0$ is such that $\m O_\xi\subset\mg^*$ is closed, then the
following holds.
\begin{itemize}
\item[(1)] $\m O_\xi$ is a closed complex submanifold of $\mg^*$;
\item[(2)] $\m O^\mr_\xi$ consists of  some connected components of $\m O_\xi\cap\mg^*_0$ and is a closed submanifold of $\m O_\xi$;
\item[(3)] $(\m O_\xi, \omega, \tau)$ is a complexification of $(\m O^\mr_\xi, \omega_0)$ in the sense of Definition \ref{def:complexification}.
\end{itemize}
\end{thm}
 Note that (1) in the above theorem is a direct corollary of Theorem 2.13 in \cite{Mon-Zip55}.

\begin{proof}

We start by giving a lemma.

\begin{lem}\label{lem:real orbit ttr}
Let $\xi\in \mg^*_0$ be such that $\m O_\xi\subset\mg^*$ is closed. Then
the following holds.
\begin{itemize}
\item[(1)] $G_{0,\xi}$ is a real form of $G_\xi$,
where $G_{0,\xi}$ and $G_\xi$ are the isotropy subgroups of
$G_0$ and $G$ at $\xi$ respectively.
\item[(2)] $\m O^{\mr}_\xi$ is a totally real submanifold of $\m O_\xi$ of maximal dimension.
\end{itemize}
\end{lem}
\begin{proof}
Let $\sigma_0:G_0\ra \m O^{\mr}_\xi,\ \sigma:G\ra \m O_\xi$ be the orbit maps
given by $g\mapsto g\xi$ with differentials at the identity
$d\sigma_0:\mg_0\ra T_\xi \m O^{\mr}_\xi\subset \mg^*_0$ and $d\sigma:\mg\ra T_\xi\m O_\xi\subset\mg^*$.
We have $d\sigma_0(u)=X_u$ for $u\in\mg_0$ and $d\sigma(u)=X_u$ for $u\in\mg$.
For $u=a+ib\in\mg$ with $a,b\in\mg_0$, $X_u=0$ is equivalent to $\omega(X_u,X_v)(\xi)=\xi([u,v])=0$ for all $v\in\mg_0$.
But this is equivalent to $\xi([a,v])=\xi([b,v])=0$ for all $v\in\mg^0$.
Hence $a,b\in \ker d\sigma_0$ and $\ker d\sigma=\ker d\sigma_0\oplus i \ker d\sigma_0$
and so $G_{0,\xi}$ is a real form of $G_\xi$.
This implies that $\m O^{\mr}_\xi$ is a totally real submanifold of $\m O_\xi$ of maximal dimension.
\end{proof}

Let $\m O_\xi$ be as in the above lemma.
We now define an anti-holomorphic involution on $\m O_\xi$.
Let $\tau:\mg^*\ra \mg^*$ be the conjugation map
given by $\tau(a+ib)=a-ib,\ u, v\in \mg^*_0$.
Then $\tau$ is anti-holomorphic and is an involution, \emph{i.e.}, $\tau^2=Id$.
It is clear that $\tau(\m O_\xi)$ is also a complex submanifold of $\mg^*$.
Note that by Lemma \ref{lem:real orbit ttr} we have that
$\m O^\mr_\xi$ is a totally real submanifold
of both $\m O_\xi$ and $\tau(\m O_\xi)$ of maximal dimension,
by identity theorem for holomorphic functions we have $\m O_\xi=\tau(\m O_\xi)$,
and so $\tau$ is an anti-holomorphic involution on $\m O_\xi$.

We move forward to prove that $\m O^\mr_\xi$ is a closed submanifold of $\m O_\xi$
and hence a closed submanifold of $\mg^*_0$.
By Theorem 2.13 in \cite{Mon-Zip55} and noting that $\m O_\xi$ is closed in $\mg^*$,
it suffices to prove that $\m O^\mr_\xi$ is a closed subset of $\m O_\xi$.
Note that since a closed connected component of the set of fixed points of a compact Lie group acting
on a smooth manifold is a closed submanifold (see Theorem 5.1 in \cite{Kob95}),
any connected component of $\m O_\xi\cap\mg^*_0$, which is the fixed point set of $\tau$,
is a closed submanifold of $\m O_\xi$.
Note also that $\m O_\xi\cap\mg^*_0$ is a union of coadjoint orbits of $G_0$
and any orbit of $G_0$ contained in $\m O_\xi$ is a totally real submanifold of $\m O_\xi$ of maximal dimension by Lemma \ref{lem:real orbit ttr}, so
$\m O^\mr_\xi$ is closed in $\m O_\xi$ since its complement in $\m O_\xi$, which is a union of some orbits of $G_0$, is open in $\m O_\xi$.

Since $\omega|_{\m O^\mr_\xi}=\omega_0$ is real,
we have $\overline(\tau^*(\omega))=\omega$ on $\m O^\mr_\xi$.
By the identity theorem, we see that $\tau^*(\omega)=\bar\omega$.

\end{proof}

\section{The Density Property and Hamiltonian Carleman approximation}

By the considerations in the previous section, to prove Theorem \ref{thm:carleman} we may
consider the following framework.  We let $Z\subset\mathbb C^{N}$ be a closed connected complex
manifold equipped with a holomorphic symplectic form $\omega$.  Assume
that $Z_0$ is the union of some smooth connected components of $Z\cap\mathbb R^{N}$ with $\mathrm{dim}_{\mathbb R} Z^j_0=\mathrm{dim}_{\mathbb C}(Z)$
for each component $Z^j_0$,
and such that $\omega_0:=\omega|_{Z_0}$
is a real symplectic form on $Z_0$. Letting $u_j, j=1,...,N$, denote the coordinates on $\mathbb C^{N}$,
we assume further that for any linear function $u=\sum_{j=1}^N c_j\cdot u_j$ the vector field
$X_u$ defined on $Z$ by $\iota_{X_u}\omega=\partial u$ is complete.  \

We let $\mathcal V_h(Z)$ denote the Lie algebra of holomorphic Hamiltonian vector fields on $Z$, and
we let $\mathcal V_h^I(Z)$ denote the Lie sub-algebra of $\mathcal V_h(Z)$ generated by complete
vector fields.

\subsection{The density property for $\mathcal V_h(Z)$}

The following lemma follows easily from the assumptions above and standard arguments in Anders\'{e}n-Lempert theory.

\begin{lem}\label{lem:density}
In the setting above we have that $\mathcal V_h^I(Z)$ is dense in $\mathcal V_h(Z)$.
Moreover, if $A\subset\mathbb R^n$ is compact, and if $X_y$ is a continuous family
of smooth vector fields on $Z_0$ with
$X_y\in\mathcal V_h(Z_0)$ for all $y\in A$, then
if $\epsilon>0$, $r\in\mathbb N$, and if $K\subset Z_0$ is compact, there exists
a continuous family of complete holomorphic vector fields $Y_{y,1},...,Y_{y,n}\in\mathcal V_h^I(Z)$, all tangent to $Z_0$, such that
$$
\|X_y-\sum_{j=1}^n Y_{y,j}\|_{C^r(K)}<\epsilon.
$$
\end{lem}
\begin{proof}
We write $u_j=x_j+iy_j$.
We start with the case of $X_y \in \mathcal V_h(Z_0)$ for each fixed $y\in A$.  Let
$u_y$ be a family of potentials for $X_y$, continuous in $C^{r+1}$-norm in the $(x)$-variables
with respect to $y$. By Weierstrass' Approximation Theorem we may approximate $u$ arbitrarily
well in $C^{r+1}$-norm by polynomials $P$ in $x$ depending continuously on $y$, and so for the
purpose of approximating $X$ we will consider the vector fields $\tilde X_y$ defined on $Z_0$ by
$dP_y=\iota_{\tilde X_y}\omega_0$. First we let $P_y^o$ denote the polynomials $P_y$
extended in the obvious way to polynomials in the variables $u_j$, and we let $\tilde X_y^o$
denote the holomorphic vector fields defined on $Z$ by $\partial P_y^o=\iota_{{\tilde X}^o_y}\omega$.

We will first show that $\tilde X_y^o$ is tangent to $Z_0$. For any point $\zeta\in Z_0$, since
$Z_0$ is real analytic, there exists a real holomorphic embedding $g:U\rightarrow Z$
with $g(\mathbb R^{2n}\cap U)=Z_0\cap g(U)$, where $U$ is a neighbourhood of the origin
in $\mathbb C^{2n}$ and $g(0)=\zeta$. Setting $\widehat\omega=g^*\omega$ and
$\widehat P_y = g^* P^o_y$ it suffices to show that the vector fields $\widehat X_y$
defined by $\partial\widehat P_y =\iota_{\widehat X_y}\widehat\omega$ is tangent
to $\mathbb R^{2n}$. Note that $\widehat P_y$ and $\widehat\omega$ are real.
In particular
$$
\omega(z)=\sum_{i<j} a_{ij}(z)dz_i\wedge dz_j
$$
where all $a_{ij}$ are real holomorphic functions. Furthermore, we have that
$$
\alpha_y(z)=\partial\widehat P_y(z) = \sum_{j=1}^{2n} b_j(z)dz_j,
$$
with $b_j(z)$ real holomorphic functions, so it is straight forward to see that $\widehat X_y$
are all real holomorphic vector fields.

\medskip

%

We may now write
\begin{equation}
P^o_y(u)=\sum_{|\alpha|\leq N} a_{y,\alpha}u^\alpha,
\end{equation}
where the $a_{y,\alpha}$'s are real valued continuous functions on $A$.

It is a fundamental result in Anders\'{e}n-Lempert Theory that
\begin{equation}
u^\alpha =\sum_{k=1}^{M_\alpha} b^\alpha_k\cdot (c^\alpha_{k,1}\cdot u_1 + \cdot\cdot\cdot + c^\alpha_{k,N}\cdot u_N)^{|\alpha|},
\end{equation}
where all coefficients may be taken to be real. Hence, we have that $P_{y}(u)$ is a sum of real polynomials of the form
\begin{equation}
g^\alpha_{y,k}(u)=d^\alpha_{y,k}\cdot (c^\alpha_{k,1}\cdot u_1 + \cdot\cdot\cdot + c^\alpha_{k,N}\cdot u_N)^{|\alpha|}=d^\alpha_{y,k}\cdot (f^\alpha_{y,k}(u))^{|\alpha|},
\end{equation}
where by assumption $X_{f^\alpha_{y,k}}$ is complete on $Z$. Since we have that
$$
X_{(f^\alpha_{y,k})^{\alpha}} = |\alpha|(f^\alpha_{y,k})^{|\alpha|-1}X_{f^\alpha_{y,k}}
$$
we see that $X_{(f^\alpha_{y,k})^{\alpha}}$ is complete, and this concludes the proof of the lemma.
\end{proof}

\subsection{Convergence in $C^r$-norm}
For the lack of a suitable reference we will in this subsection include a result
on $C^r$-regularity of solutions of ODE's (the following lemma), and a result
on $C^r$-approximation by consistent algorithms (see Theorem \ref{thm:consistent convergenceCk} below).
We will need the results for vector fields on smooth manifolds, but since they are all local
in nature, we state and prove them in $\mathbb R^n$.

\begin{lem}\label{lem:flow convergence}
Let $D$ be an open set in $\mr^n$ and let $X^j, X:[0,1]\times D\ra\mr^n$ be smooth maps, $j\geq 1$.
We view $X^j_t:=X^j(t,\cdot)$ and $X_t=X(t,\cdot)$ as time dependent smooth vector fields on $D$.
Assume that $\phi^j, \phi:[0,1]\times D\rightarrow D$ are flows on $D$
generated by $X^j$ and $X$ with $\phi^j(0,x)=\phi(0,x)\equiv x$, namely
\begin{equation}\label{eqn:flow}
\frac{d\phi^j(t,x)}{dt}=X^j(t,\phi^j(t,x)),\ \frac{d\phi(t,x)}{dt}=X(t,\phi(t,x)).
\end{equation}
Then if $$\lim_{j\ra\infty}||X^j_t-X_t||_{C^r(K)}\ra 0$$
uniformly for all $t$ on any compact subset $K$ in $D$, then we have
$$\lim_{j\ra\infty}||\phi^j_t-\phi_t||_{C^r(K)}\ra 0$$
uniformly for all $t$ on any compact subset $K$ in $D$ (for which $\phi_t$ exists).
Here $r\geq 0$ is a fixed integer, $\phi^j_t=\phi^j(t,\cdot)$,
and $||f||_{C^r(K)}$ denotes the $C^r$-norm of $f$ on $K$,
i.e., the maximum of the $L^\infty$ norms of all partial derivatives of $f$
up to order $r$ on $K$ with respect to the variables $x$.
\end{lem}
\begin{proof}
It is a basic fact about differential equations (see e.g. Theorem 2.8 in \cite{Techl12}) that
the lemma holds in the case $r=0$ since $X_t$ and the $X^j_t$s are assumed to be smooth (in particular Lipschitz).
It is also a fact (however not as basic) that since $X_t$ (resp. $X^j_t$) is smooth, we have that $\phi(t,x)$ (resp. $\phi^j(t,x)$) is smooth (see e.g. Theorem 2.10 in  \cite{Techl12}).\

We will proceed by induction on $r$, and as induction hypothesis $(I_r)$ we will assume that the theorem holds for some $r-1$
with $r\geq 1$.  As just pointed out we have that $(I_1)$ holds. \

Letting $A(t,x)$ (resp. $A^j(t,x)$) denote
the Jacobian of $X_t$ (resp. $X^j_t$), and using the chain rule
and the equality of mixed partials in \eqref{eqn:flow}, we see that
$\frac{\partial\phi}{\partial x_i}(t,x)$ (resp. $\frac{\partial\phi^j}{\partial x_i}(t,x)$)
is a solution to the initial value problem (variational equation)
\begin{equation}
\overset{\cdot}{y} = A(t,\phi(t,x))\cdot y
\end{equation}
(resp. $\overset{\cdot}{y} = A^j(t,\phi^j(t,x))\cdot y$) with initial value $x_0=(0,...,1,...0)$ with
the $1$ at the ith spot. Now $A(t,\phi(t,x))$ (resp. $A^j(t,\phi^j(t,x))$) is smooth, and
since
$$
A^j(t,\phi^j(t,x))\rightarrow A(t,\phi(t,x))
$$
in the $(r-1)$-norm as $j\rightarrow\infty$, it follows by the induction hypothesis
that $\frac{\partial\phi^j}{\partial x_i}(t,x)\rightarrow \frac{\partial\phi}{\partial x_i}(t,x)$
in $(r-1)$-norm as $j\rightarrow\infty$. \
\end{proof}

\subsection{Consistent algorithms and $C^r$-norms}

Let $D\subset\mathbb R^n$ be a domain, let $A$ be a compact subset of $\mathbb R^m$, and let
$X(y,x):A\times D\rightarrow\mathbb R^n$ be a
smooth map, which for each fixed $y$ we interpret as a vector field on $D$. Let $\phi_{t,y}$ denote
the phase flow of $X_y$. A \emph{consistent algorithm} for $X$ is
a smooth map $\psi:I\times A\times D\rightarrow\mathbb R^n$ (here $I\subset\mathbb R$ is an unspecified interval
containing the origin) such that
\begin{equation}
\frac{d}{dt}|_{t=0}\psi_{y,t}(x)= X_y(x).
\end{equation}
The following is a basic result on approximation of flows by means of consistent
algorithms.
\begin{thm}\label{thm:consistent convergence}
With notation as above, let $\psi_{y,t}(x)$ be a consistent algorithm for $X$. Let
$K\subset D$ be a compact set, let $T>0$, and assume that the flow $\phi_{y,t}(x)$
exists for all $x\in K, y\in A$ and all $0\leq t\leq T$. Then for each $(t,x)\in I_T\times A\times K$
($I_T=[0,T]$) we have that $\psi_{y,{t/n}}^n(x)$ exists for all sufficiently large $n$.
Morover $\psi_{y,{t(y)/n}}^n(x)\rightarrow \phi_{y,t(y)}(x)$ uniformly as $n\rightarrow\infty$,
where $t:A\rightarrow I_T$ is a continuous function.
\end{thm}

In \cite{AbrahamMarsden}, Theorem 2.1.26, this was stated and proved without the uniformity in $(t,x)$,
and without the parameter space $A$. Although the proof in our case is exactly the same, we include it
for completeness.
\begin{proof}
Fix a constant $C>0$ such that the following holds.
\begin{itemize}
\item[(1)] $\|\psi_{y,t}(x) - \phi_{y,t}(x)\|\leq C\cdot t^2$, and
\item[(2)] $\|\phi_{y,t}(x)-\phi_{y,t}(x')\|\leq e^{C\cdot t}\cdot \|x-x'\|$,
\end{itemize}
for all $x,x'$ in an open neighbourhood $\Omega$ of the full $\phi(t,y,\cdot)$-orbit
of $K$, and for all $t$ sufficiently small. We let $\overline\Omega\subset D$.
Fix two points $x\in K, y\in A$, and for a fixed $n\in\mathbb N$ (large) we define
$$
x_k:=\psi_{y,t(y)/n}^{k}(x) \mbox{ and } y_k:=\phi_{y,k\cdot t(y)/n}(x)=\phi_{y,t(y)/n}^k.
$$
It is not a priori clear that $x_k$ is well defined even for large $n$, this will follow from the following.
Fix an initial $n\in\mathbb N$ such that $\psi_{y,T/n}$ is exists on $\Omega$ for
all $y\in A$. We claim that the following holds:

\begin{equation}\label{dist}
\|x_k-y_k\| \leq C\cdot (\frac{t(y)}{n})^2\cdot k\cdot e^{(k-1)\cdot C\cdot\frac{t(y)}{n}},
\end{equation}
for all $k\leq n$.
That this holds for $k=1$ follows immediately from (1), and to prove it for arbitrary $k\leq n$
we proceed by induction. By possibly having to increase $n$ we see from \eqref{dist} that $x_k\in\Omega$,
and so $x_{k+1}$ is well defined by the initial condition on $n$. We then get that
\begin{align*}
\|\psi_{y,t(y)/n}(x_k)-\phi_{y,t(y)/n}(y_k)\| & \leq \|\psi_{y,t(y)/n}(x_k)-\phi_{y,t(y)/n}(x_k)\|\\
&  + \|\phi_{y,t(y)/n}(x_k)-\phi_{y,t(y)/n}(y_k)\|\\
& \leq C\cdot(\frac{t(y)}{n})^2 + e^{C\cdot\frac{t(y)}{n}}\cdot C\cdot (\frac{t(y)}{n})^2\cdot k\cdot e^{(k-1)\cdot C\cdot\frac{t(y)}{n}}\\
& \leq C\cdot(\frac{t(y)}{n})^2(1 + k\cdot e^{k\cdot C\cdot\frac{t(y)}{n}})\\
& \leq C\cdot(\frac{t(y)}{n})^2\cdot(k+1)\cdot e^{k\cdot C\cdot\frac{t(y)}{n}}.\\
\end{align*}
This finishes the induction step, and we see that for sufficiently large $n$ we have that $x_k$
is well defined for $k\leq n$ independently of $x\in D$ and $y\in A$, and we have that
$\|x_k-y_k\|\leq C\cdot(\frac{T}{n})^2\cdot (k+1)\cdot e^{k\cdot C\cdot\frac{T}{n}}$.
\end{proof}

We will need the following generalisation of Theorem \ref{thm:consistent convergence}
which was proved in \cite{Schar}.

\begin{thm}\label{thm:consistent convergenceCk}
With the setup as in Theorem \ref{thm:consistent convergence} we have that
$\psi_{y,{t/n}}^n(x)\rightarrow \phi_t(x)$ uniformly in $C^r$-norm with respect to
the variables $(x)$on $I_T\times A\times K$ as $n\rightarrow\infty$,
for any fixed $r\in\mathbb N$.
\end{thm}
Since the proof of this theorem was given in the unpublished diplomarbeit \cite{Schar} we will include it here.

\begin{proof}
We start by considering the $C^1$-norm. For this we define a smooth vector
field on $\mathbb R^n\times\mathbb R^{n^2}$ by
\begin{equation}
Y_y(x,z) = (X(x),\frac{\partial X_y}{\partial x}(x)z).
\end{equation}
We claim first that
\begin{equation}
\Phi_y(t,x) = (\phi_y(t,x),\frac{\partial\phi_y}{\partial x}(t,x))
\end{equation}
is the phase flow of $Y$ with initial value $(x,\mathrm{Id})$. This follows by
using the chain rule to see that
\begin{equation}
\frac{d}{dt}\frac{\partial\phi_y}{\partial x}(t,x) = \frac{\partial}{\partial x}\frac{d\phi_y}{dt}(t,x)
=\frac{\partial (X_y\circ\phi_y)}{\partial x}(t,x) = \frac{\partial X_y}{\partial x}(\phi_y(t,x))\frac{\partial\phi_y}{\partial x}(t,x).
\end{equation}

\

The next step is to write down a convenient consistent algorithm for $Y_y$. We have that
\begin{equation}
\frac{d}{dt}|_{t=0}\frac{\partial\psi_{y,t}(x)}{\partial x} = \frac{\partial}{\partial x}\frac{d}{dt}|_{t=0}\psi_{y,t}(x) = \frac{\partial X_y}{\partial x}(x).
\end{equation}
Therefore, the map
\begin{equation}
\Psi_y(t,x,y)=(\psi_y(t,x),\frac{\partial\psi_y}{\partial x}(t,x)z)
\end{equation}
is a consistent algorithm for the vector field $Y_y$, and so we have that
$$
\lim_{n\rightarrow\infty}\Psi_{y,t/n}^n \rightarrow \Phi_{y,t}
$$
uniformly as $n\rightarrow\infty$.  \

We will now show that the second component of $\Psi_{y,t/n}^n(x,\mathrm{Id})$
is equal to $\frac{\partial\psi^n_{y,t/n}}{\partial x}(x)$,
from which it will follow that
$$
\frac{\partial\psi^n_{y,t/n}}{\partial x}(x)\rightarrow\frac{\partial\phi_y}{\partial x}(t,x)
$$
uniformly as $n\rightarrow\infty$ (by convergence of consistent algorithms in $C^0$-norm).
We will show this by induction.

\medskip

Note first that
$$
\Psi_{y,t/n}(x,y) = (\psi_y(t/n,x),\frac{\partial\psi_y}{\partial x}(t/n,x)y),
$$
so if we set $y=\mathrm{Id}$, we see that the second
component of $\Psi_{y,t/n}^1(x,\mathrm{Id})$
is equal to $\frac{\partial\psi^1_{y,t/n}}{\partial x}(x)$. As our
induction hypothesis $I_m$ ($1\leq m<n$) we now assume that
the second component of $\Psi_{y,t/n}^m(x,\mathrm{Id})$
is equal to $\frac{\partial\psi^m_{y,t/n}}{\partial x}(x)$ (as we have seen $I_1$ holds).
We then have that
$$
\Psi^m_{y,t/n}(x,\mathrm{Id})=(\psi^m_{y,t/n}(x),\frac{\partial\psi^m_{y,t/n}}{\partial x}(x).
$$
We get that
$$
\Psi^{m+1}_{y,t,n}(x,\mathrm{Id})=(\psi^{m+1}_{y,t/n}(x),\frac{\partial\psi_{y,t/n}}{\partial x}(\psi^{m}_{y,t/n}(x))\frac{\partial\psi^m_{y,t/n}}{\partial x}(x)).
$$
And on the other hand, by the chain rule we have that
$$
\frac{\partial}{\partial x}\psi_{y,t/n}(\psi^m_{y,t/n}(x)) = \frac{\partial\psi_{y,t/n}}{\partial x}(\psi^n_{y,t/n}(x))\frac{\partial\psi^m_{y,t/n}}{\partial x}(x)).
$$
This completes the proof of convergence in $C^1$-norm, and the general case follows by induction on $r$.

\end{proof}

For the following corollary we introduce some notation. Let $X:I\times D\rightarrow\mathbb R^n$
be a smooth map which we interpret as a time dependent vector field on a domain $D\subset\mathbb R^n$.
Fix an $n\in\mathbb N$. For each $j=0,1,2,...,n-1$ and $t\in I_T$ we let
$X^{jt/n}(x)$ denote the time independent vector field $X_{(tj)/n}(x)$, and we let $\psi^{jt/n}_{s}$ denote its flow.
Finally, on any compact set $K\subset D$ and any for all $t$ where the following is well defined, we set
\begin{equation}\label{comp}
\psi^n_{t}:=\psi^{(n-1)t/n}_{t/n}\circ \psi^{(n-2)t/n}_{t/n}\circ\cdot\cdot\cdot \circ\psi^{t/n}_{t/n}\circ \psi^0_{t/n}.
\end{equation}

We have the following corollary to Theorem \ref{thm:consistent convergenceCk}.
\begin{cor}\label{cor:freezetime}
Let $X:I\times D\rightarrow\mathbb R^n$
be a smooth map which we interpret as a time dependent vector field on a domain $D\subset\mathbb R^n$,
and let $\phi_t$ denote its phase flow. Let $K\subset D$ be a compact set and let $I_T$ be an interval
such that $\phi_{t}(x)$ exists on $I_T\times K$. Then $\psi^n_{t}(x)\rightarrow \phi_{t}(x)$
uniformly and in $C^k$-norm in the variables $(x)$ on $I_T\times K$ as $n\rightarrow\infty$ (where $\psi^n_t$ is defined as in \eqref{comp}).
\end{cor}
\begin{proof}
We consider the time independent vector field $Y$ on the orbit $(t,\phi_{t}(D)$ by
$$
Y(t,x):=(1,X(t,x)).
$$
Then the flow of $Y$ is given by
$$
\Phi_{s}(t,x):=(t+s,\phi_{t+s}(\phi_{t}^{-1}(x))).
$$
Now for each $t$ we let $\psi^t_{s}$ denote the flow of the time independent vector field $X_{t}$
where $t$ is fixed, and we define
$$
\Psi_{s}(t,x)=(t+s,\psi^t_{s}(x)).
$$
Then $\Psi_{s}$ is a consistent algorithm for $Y$, and so $\Psi_{s/n}^n\rightarrow \Phi_s$
uniformly on $I\times K$ in $C^k$-norm as $n\rightarrow\infty$. Now starting from $t=0$ and
$x\in K$ it is easy to see by induction on $j$ that
\begin{equation}
\Psi^j_{s/n}(x)=(js/n,\psi^{(j-1)s/n}_{s/n}\circ\cdot\cdot\cdot\circ\psi^0_{s/n}(x)),
\end{equation}
for $j\leq n$, and for $j=n$ we see that the second component gives \eqref{comp} after substituting $t$ for $s$.
\end{proof}

\subsection{An Anders\'{e}n-Lempert type Theorem}
With the setup from the beginning of this section, we now prove a theorem that will be the key ingredient in the proof
of Theorem \ref{thm:carleman}.

\begin{thm}\label{thm:al}
Assume that all connected components of $Z_0$ are simply connected and, and let $\psi:[0,1]\times Z_0\rightarrow Z_0$ be a smooth map such that $\psi_t\in\mathrm{Symp}(Z_0,\omega_0)$
for each $t\in [0,1]$. Assume further that $\psi_0=\mathrm{id}$, and that there exists an $R\geq 0$ such that
$\psi_t(x)=x$ for all $\|x\|\leq R$, for all $t\in [0,1]$. Then for any $r\in\mathbb N$ and $a>0, 0<b<R$, there exists a sequence $\Psi_{j,t}:Z\rightarrow  Z$
with $\Psi_{j,t}\in\mathrm{Symp}(Z,\omega)$ for all $j\in\mathbb N, t\in [0,1]$, and with $\Psi_{j,t}$ $Z_0$-invariant, such that
\begin{equation}
\lim_{j\rightarrow\infty} \|\Psi_{j,t}-\psi_t\|_{C^r(a\overline{\mathbb B_{\mathbb R}^{N}}\cap Z_0)} = 0
\end{equation}
and
\begin{equation}
\lim_{j\rightarrow\infty} \|\Psi_{j,t}-\mathrm{id}\|_{C^r(b\overline{\mathbb B^{N}}\cap Z)} = 0
\end{equation}
uniformly in $t$.
\end{thm}

\begin{proof}
We may assume that $a>b$. Let $K\subset Z_0$ be a compact set containing the entire $\psi_t$-orbit of
$a\overline{\mathbb B_{\mathbb R}^{N}}\cap Z_0$ in its interior.  \

{\bf Step 1:} Define the time dependent vector field
$$
X_t(\psi_t(x)):=\frac{d}{dt}\psi_t(x).
$$
Then $\psi_t$ is the phase flow of $X$. Since $\psi_t(x)=x$ for all $\|x\|\leq R$ and $t\in [0,1]$
we may extend $X_t$ to be identically zero on an open neighbourhood of $b\overline{\mathbb B^{2n}}$.
Consider $\psi^n_t$ as in \eqref{comp}. By using Cartan's formula one sees that each $X_{jt/n}$
is Hamiltonian, hence each $\psi^{jt/n}_{s}$ is symplectic. By Corollary \ref{cor:freezetime} we have that
$\psi^n_t\rightarrow\psi_t$ uniformly in $C^r$-norm with respect to the variables $(x)$ uniformly in $t$.
Hence, fixing a large enough $n$ it remains to prove that each $t$-parameter family of flows $\psi^{jt/n}_s$,
is approximable in $C^r$-norm in the variables $(x)$ by a $t$-parameter family of flows $\tilde\psi^{j,t}_s$
uniformly in $(t,s)$, where $\tilde\psi^{j,t}_s\in \mathrm{Symp}(Z,\omega)$ for each fixed $t$, and leaving
$Z_0$ invariant.  By Lemma \ref{lem:flow convergence} it suffices to approximate the $t$-parameter
family $X_{jt/n}$ uniformly in $C^r$-norm in the variables $(x)$ uniformly in $t$, by complete holomorphic vector fields
which are tangent to $Z_0$. \

{\bf Step 2:} \

The approximation of the family $X_{jt/n}$ on $K$ is now immediate from Lemma \ref{lem:density}, we
just have to take some extra care to achieve that the approximation is uniformly close to the identity on
$b\overline{\mathbb B^{N}}\cap Z$. Let $P_{j,t}$ denote the real potential for $X_{jt/n}$ on $Z_0$ extended
to a smooth function on $\mathbb R^N$. Then $P_{j,t}$ may be chosen to be zero for $\|x\|\leq R$, and so
$P_{j,t}$ may be extended to be zero on $R\cdot\mathbb B^N_{\mathbb C}$. Hence (by possibly having to decrease $R$
slightly) since $R\cdot\mathbb B^N_{\mathbb C}\cup\mathbb R^N$ is polynomially convex, we may
approximate $P_{j,t}$ to arbitrary precision on $R\cdot\mathbb B^N_{\mathbb C}\cup K$ by a
parameter family $Q_{j,t}$ of holomorphic polynomials, and by setting
$\tilde Q_{j,t}(u):=\frac{1}{2}(Q_{j,t}(u)+\overline{Q_{j,t}(\overline u)})$ we obtain a real
holomorphic polynomial approximating $Q_{j,t}$, and we get that $X_{\tilde Q_{j,t}}$
approximates $X_{jt/n}$ on $K$ and the identity on $R\cdot\mathbb B^N_{\mathbb C}$.
Following the proof of Lemma \ref{lem:density}
we obtain complete Hamiltonian vector fields $Y_{t,k}$ on $Z$ such that
$\sum_{k=1}^M Y_{t,k}=X_{\tilde Q_{j,t}}$.  \

{\bf Step 3:} \

Let $\sigma^{t,k}_s$ denote the flow of $Y_{t,k}$ for $k=1,...,M$, and set
\begin{equation}
\Psi^{t,j}_s:= \sigma^{t,M}_s\circ\cdot\cdot\cdot\circ \sigma^{t,1}_s.
\end{equation}
Then $\Psi^{t,j}_s$ is a consistent algorithm for $X_{\tilde Q_{j,t}}$, and so by Theorem \ref{thm:consistent convergenceCk}
we have that $(\Psi^{t,j}_{t/nm})^m\rightarrow \psi^{jt/n}_{t/n}$ as $m\rightarrow\infty$.

\end{proof}

\subsection{Proof of Theorem \ref{thm:carleman}}

As explained in the previous section, we may prove the Carleman approximation in
the context presented in the beginning of this section, \emph{i.e.}, for the pair $(Z,Z_0)$.
By assumption we have that $\phi$ is smoothly isotopic to the identity map, and
so there exists a smooth isotopy $\psi_t, t\in [0,1]$, of symplectomorphisms of $Z_0$
such that $\psi_0=\mathrm{id}$ and $\psi_1=\phi$.
For $R>0$ we set $Z_R=R\cdot\overline{\mathbb B^N}\cap Z$ and
$Z_{0,R}=R\cdot\overline{\mathbb B^N}\cap Z_0$.  We will construct the approximating automorphism inductively. \

Assume that we have constructed $\phi_j\in\mathrm{Symp}(Z,\Omega)$ leaving $Z_0$ invariant,
constructed an isotopy $\psi^j_t$ with $\psi^j_t\in\mathrm{Ham}(Z_0,\omega)$ for
each $t\in [0,1]$, and found $R^j_1, R^j_{2}\in\mathbb N$ with $R^j_1\geq j, R^j_{2}\geq R^j_1+1$,
such that the following hold:
\begin{itemize}
\item[($1_j$)] $\phi_j(R^j_1\cdot\overline{\mathbb B^N})\subset R^j_{2}\cdot\overline{\mathbb B^N}$,
\item[($2_j$)] $\|\phi_j-\phi_{j-1}\|_{R^{j-1}_1\overline{\mathbb B^N}}<\epsilon_j$ (for $j\geq 2$),
\item[($3_j$)] $\psi^j_0=\mathrm{id}$,
\item[($4_j$)] $\psi^j_t(x)=x$ for $\|x\|\leq R^j_{2}+\epsilon_j$ for all $t\in [0,1]$, and
\item[($5_j$)] $\|\psi^j_1\circ\phi_j-\phi\|_{C^r(x)}<\epsilon(x)$ for all $x\in Z_0$.
\end{itemize}

The inductive step is the following
\begin{equation}\label{istep}
\mbox{We may achieve ($1_{j+1}$)--($5_{j+1}$) with $\epsilon_{j+1}$ arbitrarily small.}
\end{equation}

Set $R^{j+1}_1=R^j_2+1$ and fix $S_1$ such that $Z_{0,S_1}$ contains the entire
$\psi^j_t$-orbit and $(\psi^j_t)^{-1}$-orbit of $Z_{0,R^{j+1}_1+1}$ in its interior.
Fix a cutoff function $\chi(z)$ such that $\chi(z)=0$ for all $\|z\|\leq S_1$ which
is identically equal to $1$ for $\|z\|\geq S_1+1$.
By Theorem \ref{thm:al}
we may approximate $\psi^j_t$ to arbitrary precision near $Z_{0,S_1+2}$
by $Z_0$-invariant symplectomorphisms $\Psi^j_t$, which also approximates the identity
to arbitrary precision near $R^j_2\cdot\overline{\mathbb B^N}$.  Next we set
\begin{equation}
\sigma^j_t = (\Psi^j_t)^{-1}\circ\psi^j_t,
\end{equation}
and note that we now may assume that this is as close as we like to the identity on $Z_{0,S+2}$. \

\medskip

Letting $P_t$ denote the Hamiltonian potential for $\sigma^j_t$ we may then consider
the isotopy $\tilde\sigma^j_t$ whose potential is the function $\chi\cdot P_t$. Then
we may still assume that $\tilde\sigma_t$ is is close we like to the identity on $Z_{0,S_1+2}$,
it is equal the identity on $Z_{0,S_1}$, and it is equal to $\sigma^j_t$ outside
of $Z_{0,S_1+1}$.  Note that we may now assume that $\Psi^j_t\circ\tilde\sigma^j_t$
approximates $\psi^j_t$ to arbitrary precision on $Z_0$ and also that
$(\Psi^j_t\circ\tilde\sigma^j_t)^{-1}$ approximates $(\psi^j_t)^{-1}$ to arbitrary precision on $Z_0$. \

\medskip

Next we fix $R^{j+1}_2$ such that $\Psi^j_1(R^{j+1}_1\mathbb B^N)\subset\subset R^{j+1}_2\mathbb B^N$.
Choose $T$ such that $Z_{0,T}$ contains the entire $\tilde\sigma^j_t$-orbit and  $(\tilde\sigma^j_t)^{-1}$-orbit of $Z_{0,R^{j+1}_2+1}$ in its interior,
let $\tilde\chi$ be a cutoff function such that $\tilde\chi(z)=0$ for all $\|z\|\leq T$ which
is identically equal to $1$ for $\|z\|\geq T+1$. Note that $\tilde\sigma^{j}_t$ may be extended continuously
to the identity map near $R^{j+1}_1\overline{\mathbb B^N}$, thus by arguing additionally as in the proof
of Theorem \ref{thm:al}, we may get an approximating isotopy $\tilde\Psi^j_t$ as in Theorem \ref{thm:al},
approximating $\tilde\sigma^j_t$ to arbitrary precision near $Z_{0,T+2}$, and the identity to arbitrary precision
on $R^{j+1}_1\mathbb B^N$.

\medskip

Next, we set
\begin{equation}
\widehat\sigma^j_t = \psi^j_t\circ (\Psi^j_t\circ\tilde\Psi^j_t)^{-1}
\end{equation}
We let $\widehat P_t$ denote the potential of $\widehat\sigma^j_t $, and we finally
set $\psi^{j+1}_t$ be the Hamiltonian flow on $Z_0$ whose potential is $\tilde\chi \widehat P_t$,
and note that $\psi^{j+1}_t$ is the identity near $R^{j+1}_2\overline{\mathbb B^N}$.
Finally, setting $\phi_{j+1}:=\Psi^j_1\circ \tilde\Psi^j_1\circ\phi_j$, we see that we have established
($1_{j+1}$)-($5_{j+1}$).  \

Finally, it is standard to construct the sequence $\phi_j$, choosing each $\epsilon_j$ sufficiently small,
such that $\phi_j$ converges to the desired approximating automorphism of $Z$, we leave
the details to the reader.

\section{Some examples}\label{sec:exm}
In this section we present some typical example of coadjoint orbits of real Lie groups and their complexifications.

\begin{exm} (The flat space) Let $G_0$ be the Heisenberg group given by
$$G_0=\{g=\left(\begin{array}{ccc}
1  & a & c\\
0  & 1 & b \\
0  & 0 & 1
\end{array}\right): a, b, c\in\mr\},$$
whose complexification $G$ is defined by the same form by requiring $a, b, c\in\mc$.
\bi
\item Coajoint orbits in $\mg^*_0$.

We can identify $\mg^*_0$ with $\mr^3$, with the coadjoint action given by
$$g\cdot(x_1,x_2,x_3)=(x_1+bx_3, x_2-ax_3, x_3).$$
We consider the coadjoint orbit $\mathcal O^\mr_x$ through $x=(x_1, x_2, x_3)$.
If $x_3=0$, $\mathcal O^\mr_x$  is a single point; if $x_3\neq 0$,
we can identify $\mathcal O^\mr_x$  with $(\mr^2, \omega_0)$ by the map
$$(x_1+bx_3,x_2-ax_3)\mapsto (x, y)\in\mr^2,$$
where $\omega_0=\frac{1}{x_3}dx_1\wedge dx_2$.
In particular, if $x_3=1$, $\mathcal O^\mr_x$  is isomorphic as symplectic manifolds to $\mr^2$ with the standard syplectic structure.

\item Coajoint orbits in $\mg^*_0$.

We can identify $\mg^*$ with $\mc^3$,
with the adjoint action of $G$ on $\mg^*$ given by the same way.
The coadjoint orbit $\mathcal O_z$ with $z=(z_1,z_2,z_3)\in\mc^3$ is a single point if $z_3=0$,
and is isomorphic to $(\mc^2,\Omega_0)$ with $\Omega_0=z_3^{-1}dz_1\wedge dz_2$ if $z_3\neq 0$.

When $z=x\in\mr^3$, $(\mathcal O_z, \Omega_0, \tau)$ is a complexification of $(\mathcal O_x, \omega_0)$,
with the antiholomorphic convolution $\tau$ given by the complex conjugate.
\ei
\end{exm}

\begin{exm}(The parabolic space) Let $G_0=SO(3)\approx SU(2)$ be the special orthogonal group.
The complexification $G$ of $G_0$ is $SO(3,\mc)$.
\bi
\item Coajoint orbits in $\mg^*_0$.

We can identify $\mg^*_0$ with $\mr^3$, with the coadjoint action given by rotations.
The coadjoint orbits are parametrized by $\mr_{\geq 0}$ and given by:
\begin{equation}\label{eqn:orbit para real case}
x_1^2+x_2^2+x_3^2=R^2,
\end{equation}
with the canonical symplectic structure given by
$$\omega=\frac{dx_2\wedge dx_3}{x_1}.$$
When $R=0$, the orbit is a single point, and when $R=1$, the orbit is isomorphic to the Riemann sphere,
with the symplectic structure given by the Fubini-Studty metric.

\item Coajoint orbits in $\mg^*$.

We identify $\mg^*=\mg^*_0\otimes \mc$ with $\mc^3$.
The origin is a closed coadjoint orbit, and other closed
coadjoint orbits are parametrized by $\mc^*$ and have the form:
\begin{equation}\label{eqn:orbit para complex case}
 z_1^2+z_2^2+z_3^2=R^2,\ (R\in\mc^*),
\end{equation}
with the canonical holomorphic symplectic structure given by
$$\Omega=\frac{dz_2\wedge dz_3}{z_1}.$$

When $R\in \mr_{>0}$, the coajoint orbit of $G$ given by \eqref{eqn:orbit para complex case}
is a complexfication of the coadjoint orbit of $G_0$ given by \eqref{eqn:orbit para real case},
with the antiholomorphic convolution $\tau$ given by the complex conjugate.
\ei
\end{exm}

\begin{exm}(The hyperbolic space)
Let $G_0=SO(2,1)\approx SL(2,\mr)$, whose complexification is $G=SO(2,1,\mc)\approx SL(2,\mc)$.

The coadjoint action is equivalent to the standard action of $SO(2,1)$ on $\mr^3$.
The orbits are parametrized by $\mr$ and have the following form:
$$x_3^2-x_1^2-x_2^2=R.$$
The are divided into 3 types corresponding to $R>0, R=0$ and $R<0$.
Here we just consider the case that $R>0, x_3>0$.
For this case, the isotropy group of the coadjoint action at $(0,0, R)$ is $S^1$,
and the coadjoint orbit is given by $SL(2,\mr)/S^1$,
which is isomorphic to the upper half plane $\mathbb H$ with the symeplectic structure given by
the Poincar\'{e} metric.

The coadjoint orbits in $\mg^*\approx\mc^3$ are given by the same equations.
They are complexifications of coadjoint orbits of $G_0$,
with the antiholomorphic convolution $\tau$ given by the complex conjugate.

\end{exm}

\bibliographystyle{amsplain}

\end{document}